\documentclass[fleqn, a4paper, oneside, doublespacing]{imsart}

\usepackage{amsthm,amsmath,amssymb}

\usepackage{graphicx}

\usepackage[authoryear]{natbib}

\bibpunct{[}{]}{;}{a}{,}{,}
\RequirePackage[OT1]{fontenc}

\startlocaldefs

\newtheorem{theorem}{Theorem}

\newtheorem{lemma}{Lemma}

\theoremstyle{remark}

\theoremstyle{definition}

\endlocaldefs

\begin{document}

\begin{frontmatter}

\title{Algebraic polynomials and moments of stochastic integrals}
\runtitle{Estimates for stochastic integrals}



\begin{aug}
\author{\fnms{Mikhail}
\snm{Langovoy}\corref{}\ead[label=e1]{langovoy@eurandom.tue.nl}}


\affiliation{
         EURANDOM, Eindhoven University of Technology,\\
         Eindhoven, The Netherlands.}

\address{
EURANDOM, Technische Universiteit Eindhoven,  \\
Postbus 513, \\
Den Dolech 2,\\
5600 MB  Eindhoven, The Netherlands\\
\printead{e1}\\
Phone: (+31) (40) 247 - 8113\\
Fax: (+31) (40) 247 - 8190\\}

\runauthor{M. Langovoy}
\end{aug}

\begin{abstract}
We propose an algebraic method for proving estimates on moments of stochastic integrals. The method uses qualitative properties of roots of algebraic polynomials from certain general classes. As an application, we give a new proof of a variation of the Burkholder-Davis-Gundy inequality for the case of stochastic integrals with respect to real locally square integrable martingales. Further possible applications and extensions of the method are outlined.\\
\end{abstract}

\smallskip


\end{frontmatter}

\section{Introduction}\label{Introduction}

Connections between special algebraic polynomials and stochastic integrals have a long history (see \cite{Wiener_Chaos}, \cite{Ito_1951}), and received considerable attention in stochastic analysis (\cite{Ikeda_Watanabe}, \cite{Carlen91}, \cite{Borodin_Salminen}, \cite{Nualart_Schoutens}). Fruitful applications of special polynomials have been found in the theory of Markov processes (\cite{Kendall_1959}, \cite{Karlin_McGregor_57}), financial mathematics (\cite{Schoutens}), statistics (\cite{Diaconis_Zabell}). The book \cite{Schoutens} contains an extensive overview of this field of stochastic analysis and its applications.

In this paper, we study a different type of application of polynomials to stochastic integration. We show that not only properties of special systems of orthogonal polynomials can be used in stochastic analysis, but also that elementary properties of certain general classes of polynomials can be effectively utilized. In particular, we give a new proof of the Burkholder-Davis-Gundy inequality in Section \ref{Section2} as an illustration of our approach. Some possible extensions and further applications of our results are described in Section \ref{Section3}, were we also prove the B-D-G inequality for the case of stochastic integrals with respect to locally square integrable martingales.

\section{The main idea}\label{Section2}

We propose an algebraic proof for the following classic variation of the B-D-G inequality.

\begin{theorem}\label{BDG}
Let $b(s)$, $s \in [0, t]$, be a progressively measurable process, $b \in \mathcal{L}_{2} [0, T]$, $t \leq T$. Then for every $n \geq 2$ there exists constants $C_1 > 0$, $C_2 > 0$ such that

\begin{equation}\label{1}
C_1 \,\mathbb{E} {\biggr(\int_{0}^{t} b^2 (s)\,ds \biggr)}^n \,\leq \quad\mathbb{E} {\biggr(\int_{0}^{t} b(s)\,d W(s) \biggr)}^{2n} \,\leq \quad C_2 \,\mathbb{E} {\biggr(\int_{0}^{t} b^2 (s)\,ds \biggr)}^n\,.
\end{equation}

\noindent The constants $C_1$ and $C_2$ depend on $n$, but not on the process $b$.
\end{theorem}

The key feature of our approach is that we use general qualitative properties of roots of algebraic polynomials. The idea of proving the Burkholder-Davis-Gundy inequality via the use of Hermitian polynomials has been already explored by different authors (see, for example, \cite{McKean_Stochastic_Integrals}, pp.40-41 or  \cite{Ikeda_Watanabe}). However, the previous proofs used properties of polynomials in a different way and worked only for $n \leq 4$. Our proof is valid for general $n$ and, moreover, can be extended to give estimates for more general types of stochastic integrals.

The following two algebraic lemmas are in the core of our method.

\begin{lemma}\label{Lemma_2}
Consider real polynomials

\begin{equation}\label{9}
P_1 (z) \,=\, \sum_{k=0}^{2m} b_k z^{2k} \quad\textit{and}\quad  P_2 (z) \,=\, z \sum_{i=0}^{2 m_1} c_i z^{2i} \,,
\end{equation}


\noindent where $m_1 < m$ is integer and nonnegative, $b_k \geq 0$ for all $k$, $b_0 > 0$, $b_m > 0$, $c_i \geq 0$ for all $i$. Then there exists $0 < d_1 < d_2$ such that only for $z \in [d_1, d_2]$ one can have $P_1(z) \leq P_2(z)$, but for $z \notin [d_1, d_2]$ one has $P_1(z) > P_2(z)$.
\end{lemma}

\begin{proof} (Lemma \ref{Lemma_2}) Note first that $P_1(z)$ is symmetric, $P_1(z) \geq b_0$ for all $z \in \mathbb{R}$, and $P_1(z) \sim b_m z^{2m}$ as $z \rightarrow \infty$. Furthermore, $P_2 (-z) = - P_2 (z)$, and for $z \geq 0$ one has $P_2(z) \geq 0$, $P_2 (0) = 0$, and $deg\,P_1(z) < deg\,P_2(z)$, where $deg\,P_i$ denotes the degree of the polynomial $P_i$.

This implies that for $z < 0$ one has $P_1(z) > 0 > P_2(z)$. At $z = 0$ it holds that $P_1(0) = b_0 > 0 = P_2 (z)$. This shows that all possible solutions of the inequality $P_1(z) \leq P_2(z)$ are positive, i.e. bounded from below by a positive number $d_1$.

Since $P_1(z)/P_2(z) \rightarrow \infty$ as $z \rightarrow \infty$, it follows that for sufficiently large $z \geq z_0$ always $P_1(z) > P_2(z)$. Therefore, all possible solutions of the inequality $P_1(z) \leq P_2(z)$ lays in some interval $[d_1, d_2]$ with $d_1 > 0$ and $d_2 > 0$.

\end{proof}


\begin{lemma}\label{Lemma_3}
Consider real polynomials

\begin{equation}\label{12}
P_1 (z) \,=\, b_0 z^{2 m + 1} \quad\textit{and}\quad  P_2 (z) \,=\, z \sum_{i=0}^{m_1} c_i z^{2i} \,,
\end{equation}

\noindent where $m_1 < m$ is integer and nonnegative, $b_0 > 0$, $c_i \geq 0$ for all $i$, $c_0 > 0$, $c_{m_1} > 0$. Then there exists $d_2 > 0$ such that only for $z \in [-\infty, d_2]$ one can have $P_1(z) \leq P_2(z)$, but for $z > d_2$ one always has $P_1(z) > P_2(z)$.
\end{lemma}

\begin{proof} (Lemma \ref{Lemma_3}) The proof is analogous to the one of Lemma \ref{Lemma_2}. \end{proof}


Now we are able to prove the main theorem.

\begin{proof}(Theorem \ref{BDG}).
In the proof below we can assume that $b$ is bounded, since the general case follows by the usual truncation argument.

We denote for brevity

\[
\int_{0}^{t} b^2 (s)\,ds \,=\, \int b^2 \,ds\,, \quad\quad  \int_{0}^{t} b(s)\,d W(s)\,=\, \int b\,dW\,.
\]

\noindent Let us write for $n \geq 1$

\begin{eqnarray}\label{2}
\nonumber  \rho_{2n}(t) &=& H_{2n}\biggr(\int b^2 \,ds\,,\int b\,dW\,\biggr) \\
    &=& \sum_{0 \leq k \leq n} {(-1)}^{k} a_k {\biggr(\int b\,dW\,\biggr)}^{2n-2k} {\biggr(\int b^2 \,ds\,\biggr)}^{k} \,,
\end{eqnarray}

\noindent where $H_{2n}$ stands for the $2n-$th Hermitian polynomial defined as

\[
H_{2n} (x, y) \,:=\, \frac{{(-x)}^{n}}{n!} \, e^{y^2 / 2x} \, \frac{d^n}{dy^n} \, e^{-y^2 / 2x}
\,=\, \sum_{0 \leq k \leq n} {(-1)}^{k} a_k \, x^{k} y^{2n-2k}  \,,
\]

\noindent and it is known that

\[
a_k \,=\, \frac{1}{\,2^k k! (2n-2k)!\,} \,.
\]

\noindent Taking in (\ref{2}) the expectation of both sides and noting that $\mathbb{E}\,\rho_{2n} = 0$ (see \cite{McKean_Stochastic_Integrals}, pp. 37-38 or \cite{Ikeda_Watanabe}, pp. 150-152), we get


\begin{equation}\label{3}
\sum_{0 \leq k \leq n} {(-1)}^{k} a_k \mathbb{E} \biggr\{{\biggr(\int b\,dW\,\biggr)}^{2n-2k} {\biggr(\int b^2 \,ds\,\biggr)}^{k}\biggr\} \,=\, 0 \,.
\end{equation}


\noindent By H\"{o}lder's inequality, for all $1 \leq k \leq n$

\begin{eqnarray}\label{4}
           \mathbb{E} \biggr\{{\biggr(\int b\,dW\,\biggr)}^{2n-2k} {\biggr(\int b^2 \,ds\,\biggr)}^{k}\biggr\} & \leq &   \\
\nonumber     & \leq &  {\mathbb{E}}^{\frac{\,n-k\,}{\,n\,}} {\biggr(\int b\,dW\,\biggr)}^{2n}  {\mathbb{E}}^{\frac{\,k\,}{\,n\,}} {\biggr(\int b^2 \,ds\,\biggr)}^{n} \,.
\end{eqnarray}


\noindent \textbf{\emph{Part} I}. Consider first the case of even $n$, and let $n = 2m$ in (\ref{1}). Since $a_k \geq 0$ for all $k$, and also

\[
\mathbb{E} \biggr\{{\biggr(\int b\,dW\,\biggr)}^{2n-2k} {\biggr(\int b^2 \,ds\,\biggr)}^{k}\biggr\} \, \geq \, 0
\]

\noindent for all $k$, after throwing out from (\ref{3}) all the summands with even $k$, except for $k = 0$ and $k = n$, we get

\begin{eqnarray}\label{5}
\nonumber  a_0\, \mathbb{E} {\biggr(\int b\,dW\,\biggr)}^{2n} &-& \sum_{0 \leq 2 l + 1 \leq n} a_{2 l + 1} \mathbb{E} \biggr\{{\biggr(\int b\,dW\,\biggr)}^{2n - 2 k(l)} {\biggr(\int b^2 \,ds\,\biggr)}^{k(l)}\biggr\} \\
    &+& a_n \, \mathbb{E} {\biggr(\int b^2 \,ds\,\biggr)}^{n} \quad\quad \leq \quad\quad 0 \,,
\end{eqnarray}

\noindent where for integer $l \geq 0$ we denoted $k(l) = 2 l + 1$.

Applying inequality (\ref{4}) to (\ref{5}), we get

\begin{eqnarray}\label{6}
\nonumber  a_0\, \mathbb{E} {\biggr(\int b\,dW\,\biggr)}^{2n} &-& \sum_{0 \leq 2 l + 1 \leq n} a_{2 l + 1} {\mathbb{E}}^{\frac{\,n - k(l)\,}{\,n\,}} {\biggr(\int b\,dW\,\biggr)}^{2n} {\mathbb{E}}^{\frac{\,k(l)\,}{\,n\,}} {\biggr(\int b^2 \,ds\,\biggr)}^{n} \\
    &+& a_n \, \mathbb{E} {\biggr(\int b^2 \,ds\,\biggr)}^{n} \quad\quad \leq \quad\quad 0 \,.
\end{eqnarray}

\noindent Divide both parts of (\ref{6}) by $\mathbb{E} \bigr(\int b^2 \,ds\,\bigr)^{n}$ and put

\begin{equation}\label{7}
    z \, := \, \frac{ \,\, {\mathbb{E}}^{1/n} {\bigr(\int b\,dW\,\bigr)}^{2n} \, }{ \, {\mathbb{E}}^{1/n} {\bigr(\int b^2 \,ds\,\bigr)}^{n} \, } \,,
\end{equation}

\noindent then we obtain

\[
a_0\, z^n \,-\, \sum_{0 \leq 2 l + 1 \leq n} a_{2 l + 1} z^{n - k(l)} \,+\,a_n \, \leq \, 0\,,
\]

\noindent or equivalently

\begin{equation}\label{8}
    a_0\, z^n \, +\,a_n \, \leq \, \sum_{0 \leq 2 l + 1 \leq n} a_{2 l + 1} z^{n - k(l)} \,.
\end{equation}

Let us now put in (\ref{8}) $P_1(z) = a_0\, z^n \, +\,a_n$, $P_2(z) = \sum_{0 \leq 2 l + 1 \leq n} a_{2 l + 1} z^{n - k(l)}$. By Lemma \ref{Lemma_2}, there exists positive constants $C_1$, $C_2$ such that $0 < C_1 \leq z \leq C_2$, i.e. $C_1^n \leq z^n \leq C_2^n$, and this proves (\ref{1}) for the case of $n = 2 m$.
\\

\noindent \textbf{\emph{Part} II}. Consider now the case of odd $n$, and let $n = 2m + 1$ in (\ref{1}). Throwing away from (\ref{3}) all the summands with even $k$, except for $k = 0$, we get

\begin{equation}\label{10}
    a_0\, \mathbb{E} {\biggr(\int b\,dW\,\biggr)}^{2n} \,-\,
    \sum_{0 \leq 2 l + 1 \leq n} a_{2 l + 1} \mathbb{E} \biggr\{{\biggr(\int b\,dW\,\biggr)}^{2n - 2 k(l)} {\biggr(\int b^2 \,ds\,\biggr)}^{k(l)}\biggr\} \,\leq\, 0 \,,
\end{equation}

\noindent and analogously to (\ref{8}) we derive

\begin{equation}\label{11}
    a_0\, z^n \, \leq \, \sum_{0 \leq 2 l + 1 \leq n} a_{2 l + 1} z^{n - k(l)} \,,
\end{equation}

\noindent where $z$ is defined by (\ref{7}).

After applying Lemma \ref{Lemma_3} to $ P_1(z) = a_0\, z^n $ and $ P_2(z) = \sum_{0 \leq 2 l + 1 \leq n} a_{2 l + 1} z^{n - k(l)} $ in (\ref{11}), we obtain from (\ref{11}) that $z \leq d_2$ for some positive $d_2$. Since $n$ is odd, this implies $z^n \leq d_2^n$ and the upper bound in (\ref{1}) follows.

It remains only to prove the lower bound in (\ref{1}) for $n = 2m + 1$. In this case, we leave in (\ref{3}) only the summands with even $k$ and $k = n$, thus getting

\begin{equation}\label{13}
    \sum_{0 \leq 2k < n} a_{2k} \, \mathbb{E} \biggr\{{\biggr(\int b\,dW\,\biggr)}^{2n - 4k} {\biggr(\int b^2 \,ds\,\biggr)}^{2k}\biggr\}
    \, - \, a_n \, \mathbb{E} \biggr\{{\biggr(\int b^2 \,ds\,\biggr)}^{n}\biggr\} \, \geq \ 0 \,.
\end{equation}

\noindent Analogously to our previous derivations, this implies the inequality

\[
\sum_{0 \leq 2k < n} a_{2k} \, z^{n - 2k} \, - \, a_n \, \geq \, 0\,, \quad \textit{\emph{i.e.}}
\]

\begin{equation}\label{14}
    P(z) \, := \, \sum_{0 \leq 2k < n} a_{2k} \, z^{n - 2k} \, \geq \, a_n \,,
\end{equation}

\noindent where $z$ is again as in (\ref{7}). Since $P(z)$ is a polynomial of the form $\sum_{i=1}^{m} b_i z^{2i + 1}$, it easily follows that (\ref{14}) is equivalent to $z \geq C_1$ for some constant $C_1 = C_1 (n) > 0$. Therefore, $z^n \geq C_1^n$, and the lower bound in (\ref{1}) is proved for $n = 2m + 1$.
\end{proof}

\section{Possible extensions}\label{Section3}

As an immediate extension of the above result, we prove the following more general theorem. This will not be the most general framework where the B-D-G inequality holds or where our approach could work; we decided to stick to the stochastic integrals version for uniformity of presentation.

\begin{theorem}\label{Theorem2}
Let a real process $(M(s))_{s \geq 0}$ be a locally square integrable martingale and $\langle M \rangle$ be its quadratic variation process. Let $b(s)$, $s \in [0, t]$, be a bounded progressively measurable process, $b \in \mathcal{L}_{2}^{loc} [0, T]$, $t \leq T$. Then for every $n \geq 2$ there exists constants $C_1 > 0$, $C_2 > 0$ such that

\begin{equation}\label{15}
C_1 \,\mathbb{E} {\biggr(\int_{0}^{t} b^2 (s)\,d \langle M \rangle (s) \biggr)}^n \,\leq \quad\mathbb{E} {\biggr(\int_{0}^{t} b(s)\,d M(s) \biggr)}^{2n} \,\leq \quad C_2 \,\mathbb{E} {\biggr(\int_{0}^{t} b^2 (s)\,d \langle M \rangle (s) \biggr)}^n\,.
\end{equation}

\noindent The constants $C_1$ and $C_2$ depend on $n$, but not on the processes $b$ and $M$.
\end{theorem}

%
%

\begin{proof}(Theorem \ref{Theorem2}). Let us write again


\begin{eqnarray}\label{16}
\nonumber  \rho_{2n}(t) &=& H_{2n}\biggr(\int_{0}^{t} b^2 (s) \,d \langle M \rangle (s)\,,\int_{0}^{t} b(s)\,d M(s)\,\biggr) \\
    &=& \sum_{0 \leq k \leq n} {(-1)}^{k} a_k {\biggr(\int b \,\,d M\,\biggr)}^{2n-2k} {\biggr(\int b^2 \,d \langle M \rangle \,\biggr)}^{k} \,,
\end{eqnarray}

\noindent where $H_{2n}$ is the $2n-$th Hermitian polynomial. By Theorem 29, pp. 75-76 of \cite{Protter_2004}, we have

\[
\bigr\langle \int b \,\,d M \bigr\rangle \,=\, \int b^2 \,d \, \langle M \rangle \,.
\]

%
%
%

%

\noindent From Theorem 5.1 on p. 152 of \cite{Ikeda_Watanabe}, we again see that $\mathbb{E}\,\rho_{2n} = 0$. Therefore, the proof of Theorem \ref{BDG} can be directly transferred to the present situation.

\end{proof}

In the above proof we have used only some elementary and entirely qualitative facts about certain general types of polynomials, together with such a crude technique as simple throwing out every second term in the martingale identity (\ref{3}). Nontheless, we were able to prove a rather general Burkholder-Davis-Gundy theorem. This suggests that our approach can lead to substantially stronger results in estimation of stochastic integrals with respect to Brownian motion, if we could incorporate in our method the existing quantitative estimates for roots of Hermitian polynomials. Indeed, as was shown in \cite{Davis_1976}, for $p = 2n,$ where $n$ is an integer, the best values for $C_1$ and $C_2$ in Theorem \ref{BDG} are $l_{2n}^{2n}$ and $r_{2n}^{2n}$ respectively, where $l_{2n}^{2n}$ and $r_{2n}^{2n}$ are the smallest and the largest positive roots of the Hermite polynomial $H_{2n}.$

Moreover, equation (\ref{2}) represents only one of the many polynomial-type martingales that can be composed from stochastic integrals. Other important examples are given by the connection between integrals with respect to the Poisson process and Charlier polynomials (see \cite{Bertoin_Levy}), the connection between integrals with respect to the Gamma process and Laguerre polynomials, and the link between Sheffer polynomials and several classes of L\'{e}vy processes. Many examples of such martingales can be found in \cite{Schoutens}. It seems plausible that one could modify the method of the present paper and prove estimates for more general stochastic integrals.  \\

\noindent {\bf Acknowledgments.} Author would like to thank Andrei Borodin for introducing him to this field of research, Sergio Albeverio for insightful discussions and the two anonymous referees for their comments that have lead to the improvement of the paper. Part of this research was done when the author was at the University of G\"{o}ttingen. \\

\bibliographystyle{plainnat}
\bibliography{Stochastic_Integrals}

\begin{thebibliography}{14}
\providecommand{\natexlab}[1]{#1}
\providecommand{\url}[1]{\texttt{#1}}
\expandafter\ifx\csname urlstyle\endcsname\relax
  \providecommand{\doi}[1]{doi: #1}\else
  \providecommand{\doi}{doi: \begingroup \urlstyle{rm}\Url}\fi

\bibitem[Bertoin(1996)]{Bertoin_Levy}
Jean Bertoin.
\newblock \emph{L\'evy processes}, volume 121 of \emph{Cambridge Tracts in
  Mathematics}.
\newblock Cambridge University Press, Cambridge, 1996.
\newblock ISBN 0-521-56243-0.

\bibitem[Borodin and Salminen(2002)]{Borodin_Salminen}
Andrei~N. Borodin and Paavo Salminen.
\newblock \emph{Handbook of {B}rownian motion---facts and formulae}.
\newblock Probability and its Applications. Birkh\"auser Verlag, Basel, second
  edition, 2002.
\newblock ISBN 3-7643-6705-9.

\bibitem[Carlen and Kr{\'e}e(1991)]{Carlen91}
E.~Carlen and P.~Kr{\'e}e.
\newblock {$L\sp p$} estimates on iterated stochastic integrals.
\newblock \emph{Ann. Probab.}, 19\penalty0 (1):\penalty0 354--368, 1991.

\bibitem[Davis(1976)]{Davis_1976}
Burgess Davis.
\newblock On the {$L^{p}$} norms of stochastic integrals and other martingales.
\newblock \emph{Duke Math. J.}, 43\penalty0 (4):\penalty0 697--704, 1976.
\newblock ISSN 0012-7094.

\bibitem[Diaconis and Zabell(1991)]{Diaconis_Zabell}
Persi Diaconis and Sandy Zabell.
\newblock Closed form summation for classical distributions: variations on a
  theme of de {M}oivre.
\newblock \emph{Statist. Sci.}, 6\penalty0 (3):\penalty0 284--302, 1991.
\newblock ISSN 0883-4237.

\bibitem[Ikeda and Watanabe(1989)]{Ikeda_Watanabe}
Nobuyuki Ikeda and Shinzo Watanabe.
\newblock \emph{Stochastic differential equations and diffusion processes},
  volume~24 of \emph{North-Holland Mathematical Library}.
\newblock North-Holland Publishing Co., Amsterdam, second edition, 1989.
\newblock ISBN 0-444-87378-3.

\bibitem[It{\^o}(1951)]{Ito_1951}
Kiyosi It{\^o}.
\newblock Multiple {W}iener integral.
\newblock \emph{J. Math. Soc. Japan}, 3:\penalty0 157--169, 1951.
\newblock ISSN 0025-5645.

\bibitem[Karlin and McGregor(1957)]{Karlin_McGregor_57}
S.~Karlin and J.~L. McGregor.
\newblock The differential equations of birth-and-death processes, and the
  {S}tieltjes moment problem.
\newblock \emph{Trans. Amer. Math. Soc.}, 85:\penalty0 489--546, 1957.
\newblock ISSN 0002-9947.

\bibitem[Kendall(1959)]{Kendall_1959}
D.~G. Kendall.
\newblock Unitary dilations of one-parameter semigroups of markov transition
  operators, and the corresponding integral representations for markov
  processes with a countable infinity of states.
\newblock \emph{Proc. London Math. Soc.}, 3\penalty0 (9):\penalty0 417--431,
  1959.

\bibitem[McKean(1969)]{McKean_Stochastic_Integrals}
H.~P. McKean, Jr.
\newblock \emph{Stochastic integrals}.
\newblock Probability and Mathematical Statistics, No. 5. Academic Press, New
  York, 1969.

\bibitem[Nualart and Schoutens(2000)]{Nualart_Schoutens}
David Nualart and Wim Schoutens.
\newblock Chaotic and predictable representations for {L}\'evy processes.
\newblock \emph{Stochastic Process. Appl.}, 90\penalty0 (1):\penalty0 109--122,
  2000.
\newblock ISSN 0304-4149.
\newblock \doi{10.1016/S0304-4149(00)00035-1}.
\newblock URL \url{http://dx.doi.org/10.1016/S0304-4149(00)00035-1}.

\bibitem[Protter(2004)]{Protter_2004}
Philip~E. Protter.
\newblock \emph{Stochastic integration and differential equations}, volume~21
  of \emph{Applications of Mathematics (New York)}.
\newblock Springer-Verlag, Berlin, second edition, 2004.
\newblock ISBN 3-540-00313-4.
\newblock Stochastic Modelling and Applied Probability.

\bibitem[Schoutens(2000)]{Schoutens}
Wim Schoutens.
\newblock \emph{Stochastic processes and orthogonal polynomials}, volume 146 of
  \emph{Lecture Notes in Statistics}.
\newblock Springer-Verlag, New York, 2000.
\newblock ISBN 0-387-95015-X.

\bibitem[Wiener(1938)]{Wiener_Chaos}
Norbert Wiener.
\newblock The {H}omogeneous {C}haos.
\newblock \emph{Amer. J. Math.}, 60\penalty0 (4):\penalty0 897--936, 1938.
\newblock ISSN 0002-9327.

\end{thebibliography}


\end{document}